\documentclass[10pt,a4paper]{article}
\usepackage[]{graphicx}
\usepackage[]{amsmath,amssymb}
\usepackage[]{amsthm,enumerate}
\usepackage{verbatim,bm}
\usepackage{color}
\usepackage{here}

\usepackage{a4wide}
\usepackage{framed}
\usepackage{color}
\usepackage[dvipsnames]{xcolor}

\newtheorem{theorem}{Theorem}[section]
\newtheorem{lemma}[theorem]{Lemma}

\theoremstyle{definition}
\newtheorem{definition}[theorem]{Definition}

\newcommand{\defn}[1]{{\em #1}}
\theoremstyle{remark}

\newtheorem{problem}[theorem]{Problem}

\newcommand{\dd}{\; \mathrm{d}}
\newcommand\sbs{\subseteq}
\newcommand\ip[2]{\langle#1,#2\rangle}

\newcommand\nn{{\mathbb N}}

\newcommand\al{\alpha}
\newcommand\be{\beta}

\newcommand\Om{\Omega}

\newcommand\opk[1]{\mathop{\mathrm{#1}}\nolimits}
\newcommand\Harm{\opk{Harm}}
\newcommand\Hom{\opk{Hom}}

\begin{document}

\title{Semidefinite programming bounds for complex spherical codes}
\author{ Wei-Jiun Kao\thanks{W.-J. Kao was supported by National Center for Theoretical Sciences (No. 111L104040).}
\\
  National Center for Theoretical Sciences, Mathematics Division, \\
   Taipei, Taiwan\\
  {\tt wjkao@ncts.ntu.edu.tw}
\and
 Sho Suda\thanks{S.S. was supported by JSPS KAKENHI; grant number: 15K21075, 18K03395.} \\
 Department of Mathematics, \\  National Defense Academy of Japan, \\ Yokosuka, 239-8686, Japan \\
 \tt{ssuda@nda.ac.jp}
  \and
 Wei-Hsuan Yu\thanks{W.-H. Yu was supported by  MOST  in  Taiwan  under  Grant  MOST109-2628-M-008-002-MY4.}
 \\
  Mathematics Department, \\
  National Central University, \\
   Taoyuan, Taiwan\\
  {\tt u690604@gmail.com}
}
\date{\today}

\maketitle
\begin{abstract}
A complex spherical code is a finite subset on the unit sphere in $\mathbb{C}^d$. A fundamental problem on complex spherical codes is to find upper bounds for those with prescribed inner products. In this paper, we determine the irreducible decomposition under the action of the one-point stabilizer of the unitary group $U(d)$ on the polynomial ring $\mathbb{C}[z_1\ldots,z_d,\bar{z}_1,\ldots,\bar{z}_d]$ in order to obtain the semidefinite programming bounds for complex spherical codes.      
\end{abstract}

\section{Introduction}
In \cite{DGS}, (real) spherical codes and designs were introduced. 
A spherical code is a finite subset of the unit sphere, denoted $S^{d-1}$, in $\mathbb{R}^d$, and a real spherical design is a finite subset of $S^{d-1}$.    
Finding several bounds for those are obtained by the linear programming. 
This method is based on the irreducible decomposition of the polynomial ring $\mathbb{R}[z_1,\ldots,z_d]$ by the action of the orthogonal group $O(d)$. 
The LP method provides the best possible upper bounds for spherical codes, such as spherical codes with the same inner products with the normalized minimum vectors to lie on the unit sphere of the $E_8$ root lattice or the Leech lattice.   

In \cite{BV}, the semidefinite programming was developed and succeeded to give another proof of the kissing number problem in $\mathbb{R}^4$ \cite{M}.   
Since then, the SDP method have been extensively used to give upper bounds for spherical codes with few distances \cite{Barg13, Liu}.

Complex spherical codes and designs were defined and studied in \cite{RS} with intention to application to complex MUBs or SIC-POVMs. 
The LP method is crucial in this case as well, which is based on the  irreducible decomposition of the polynomial ring $\mathbb{C}[z_1,\ldots,z_d,\bar{z}_1,\ldots,\bar{z}_d]$ by the action of the unitary group $U(d)$. 

In this paper we consider the action of the one-point stabilizer, which is identified with $U(d-1)$, with respect to a point in the unit sphere of the unitary group $U(d)$ on the polynomial ring $\mathbb{C}[z_1,\ldots,z_d,\bar{z}_1,\ldots,\bar{z}_d]$ in order to obtain upper bounds for complex spherical codes by the semidefinite programming. 

The organization of this paper is as follow. 
In Section~\ref{sec:pre}, we review complex spherical codes and association schemes. 
In Section~\ref{sec:SDP}, we determine the irreducible decomposition under the action of $U(d-1)$ on $\mathbb{C}[z_1\ldots,z_d,\bar{z}_1,\ldots,\bar{z}_d]$. 
Note that this action is not multiplicity free. 
We associate each isotypic component, that is the sum of the isomorphic  irreducible subspaces, with so called the zonal matrix.  
We will find the zonal matrix explicitly with respect to some basis in a similar way in \cite{BV}. 
We then formulate the semidefinite programming in Theorem~\ref{thm:sdpbound}. 
As an application, we demonstrate how we improve the known upper bound for complex $2$-codes in Section~\ref{sec:2code}. 
In Section~\ref{sec:complex3code}, we show non-existence results for complex $3$-codes, which show that there exist no commutative association schemes in \cite{MW}. 
In Section~\ref{sec:real3code}, we deal with real spherical $3$-distance sets related to Hadamard matrices and association schemes. We improve the known upper bounds obtained by the linear programming in \cite{AHS} and reprove a result in association schemes.    

\section{Preliminaries}\label{sec:pre}
\subsection{Complex spherical codes}
In this paper, vectors are denoted column vectors. 
Let $\Omega(d)=\{x\in \mathbb{C}^d\mid  x^*x=1\}$ denote the unit sphere in $\mathbb{C}^d$. 
Given $X \sbs \Omega(d)$, we define the \defn{inner product set} $A(X)$ to be 
\[
A(X) := \{a^*b \mid  a, b \in X, a \neq b \}.
\]
A polynomial $F(x)\in\mathbb{R}[x,\bar{x}]$ is said to be an \defn{annihilator polynomial} of $X$ if 
$F(\al)=0$ for each $\al\in A(X)$ and $F(1)$ is positive. 

\begin{definition}
A finite set $X$ in $\Omega(d)$ is said to be \emph{complex code}. When $|A(X)| = s$, a complex code $X$ is said to be a  \defn{complex spherical code of degree $s$} or \emph{complex $s$-code}.
\end{definition}
A complex $s$-code $X$ is \emph{distance invariant} if the size of a set $\{y\in X\mid x^* y=\alpha\}$ for $\alpha\in A(X)$ depends only on $\alpha$ and does not depend on the choice of $x$. We say the size of $\{y\in X\mid x^* y=\alpha\}$ as the \emph{valency} when $X$ is distance invariant.  

We denote by $\Hom_d(k,l)$ the vector space generated by homogeneous polynomials of degree $k$ in variables $\{z_1,\ldots,z_d\}$ and degree $k$ in variables $\{\bar{z}_1,\ldots,\bar{z}_d\}$.
The unitary group $U(d)$ acts on $\Hom_d(k,l)$, 
and the irreducible decomposition is 
\begin{align*}
\Hom_d(k,l)=\bigoplus_{m=0}^{\min(k,l)}\Harm_d(k-m,l-m).  
\end{align*}

Define an inner product on functions $f$ and $g$ on $\Om(d)$ as follows:
\[
\ip{f}{g} := \int_{\Om(d)} \overline{f(z)}g(z) \dd z.
\]
Here $\mathrm{d}z$ is the unique invariant Haar measure on $\Om(d)$, normalized so that $\int_{\Om(d)}\mathrm{d}z = 1$. With respect to this inner product, $\Harm_d(k,l)$ is orthogonal to $\Harm_d(k',l')$ whenever $(k,l) \neq (k',l')$.
Let $\nn$ denote the set of nonnegative integers. For each $(k,l) \in \nn^2$, we define a Jacobi polynomial $g_{k,l}^d$ as follows:
\[
g_{k,l}^d(x) := \frac{m_{k,l}^d(d-2)!k!l!}{(d+k-2)!(d+l-2)!} \sum_{r=0}^{\min\{k,l\}} (-1)^r \frac{(d+k+l-r-2)!}{r!(k-r)!(l-r)!}x^{k-r}\overline{x}^{l-r}, 
\]
where 
\begin{align*}
m_{k,l}^d=\dim(\Harm_d(k,l)) = \binom{d+k-1}{d-1}\binom{d+l-1}{d-1} - \binom{d+k-2}{d-1}\binom{d+l-2}{d-1}.
\end{align*}


Explicitly, the first few Jacobi polynomials are 
\begin{align*}
g_{0,0}^d(x) & = 1, \\
g_{1,0}^d(x) & = d x,  \displaybreak[0]\\
g_{1,1}^d(x) & = (d+1)(d x\overline{x} - 1), \displaybreak[0]\\
g_{2,0}^d(x) & = \frac{d(d+1)}{2}x^2,  \displaybreak[0]\\
g_{2,1}^d(x) & = \frac{d(d+2)}{2}((d+1)x^2\overline{x} - 2x),  \displaybreak[0]\\
g_{2,2}^d(x) & = \frac{d(d+3)}{4}((d+1)(d+2)x^2\overline{x}^2 - 4(d+1)x\overline{x} + 2),  \displaybreak[0]\\
g_{3,0}^d(x) & = \frac{d(d+1)(d+2)}{6}x^3.
\end{align*}
The Jacobi polynomials are normalized so that $g_{k,l}^d(1) = \dim(\Harm_d(k,l))$. 

Recursively, the Jacobi polynomials satisfy
\[
x g_{k,l}^d(x) = a_{k,l}g_{k+1,l}^d(x)+b_{k,l}g_{k,l-1}^d(x)
\]
where $a_{k,l}=\frac{k+1}{d+k+l}$, $b_{k,l}=\frac{d+l-2}{d+k+l-2}$ and set $g_{k,l}^d(x)=0$ unless $(k,l) \in \nn^2$. 

The essential property of the Jacobi polynomials is the following theorem, known as Koornwinder's addition theorem.

\begin{theorem}\label{thm:addition}
Let $\{e_1,\ldots,e_{m^d_{k,l}}\}$ be an orthonormal basis for the space $\Harm_d(k,l)$. Then for any $a,b \in \Om(d)$,
\[
\sum_{i=1}^{m^d_{k,l}} \overline{e_i(a)}e_i(b) = g_{k,l}^d(a^*b).
\]
\end{theorem}

The following lemma  follows from Theorem~\ref{thm:addition}.
\begin{lemma}\label{lem:sum}
Let $X$ be a complex code in $\Omega(d)$.
 Then $\sum_{a,b\in X}g_{k,l}^d(a^*b)\geq 0$. 
\end{lemma}

\subsection{Association schemes}
Let $I_n,J_n$ be the identity matrix, the all-ones matrix of order $n$, respevtively. 
We omit its subscript when the order is clear. 
Let $X$ be a finite set and $\{R_0,R_1,\ldots,R_n\}$ be a set of non-empty subsets of $X\times X$. 
Let $A_i$ denote the adjacency matrix of the graph $(X,R_i)$ for $i=0,1,\ldots, n$.
The pair $(X,\{R_i\}_{i=0}^n)$ is called a {\em commutative association
scheme} with $n$ classes if the following conditions hold:
\begin{enumerate}
\item $A_0 =I$,  
\item $\sum_{i=0}^nA_i=J$,  
\item for any $i \in \{0,1,2,\ldots,n\}$, there exists $i'\in \{0,1,2,\ldots,n\}$ such that $A_i^\top=A_{i'}$,
\item for any $i,j\in \{0,1,2,\ldots,n\}$, there exist non-negative integers $p_{i,j}^k$ such that $A_iA_j=\sum_{k=0}^n p_{i,j}^k A_k$, 
\item $A_iA_j=A_jA_i$ for any $i,j\in \{0,1,2,\ldots,n\}$. 
\end{enumerate}
If $A_i^\top=A_i$ for any $i$, then the association schemes is said to be {\em symmetric}.

The vector space $\mathcal{A}$ over $\mathbb{C}$ spanned by the matrices
$A_i$ forms an algebra.
Since $\mathcal{A}$ is commutative and semisimple, 
there exists a unique basis of $\mathcal{A}$ consisting of
primitive idempotents $E_0=\frac{1}{|X|}J_{|X|},E_1,\ldots,E_n$. 
Since the algebra $\mathcal{A}$ is closed under the entry-wise multiplication denoted by $\circ$,  we define 
the {\em Krein numbers} 
$q_{i,j}^k$ for $i,j,k\in \{0,1,\ldots,n\}$ as 
$E_i\circ E_j=\frac{1}{|X|}\sum_{k=0}^n q_{i,j}^k E_k$.
It is known that the Krein numbers are nonnegative real numbers 
(see~\cite[Lemma~2.4]{D}). 
Since both $\{A_0,A_1,\ldots,A_n\}$ and $\{E_0,E_1,\ldots,E_n\}$ form 
bases of $\mathcal{A}$, there exists a matrix $Q=(q_{ij})_{i,j=0}^n$ with
$E_i=\frac{1}{|X|}\sum_{j=0}^n q_{ji}A_j$. 
The number $k_i:=p_{i,i'}^0=p_{i0}$ is \emph{valency} of the regular graph $(X,R_i)$. 
The number $m_i=q_{i,\hat{i}}^0=q_{0i}={\rm rank} E_i$ is said to be the \emph{multiplicity}, where $\hat{i}$ is defined to be $E_{\hat{i}}={E_i}^\top$.

A commutative association scheme is {\em imprimitive} if some graph in the scheme is disconnected. 
\begin{theorem}\label{thm:imas}
The following are equivalent: 
\begin{enumerate}
    \item A commutative association scheme is imprimitive, 
    \item for some $j>0$, $E_j$ has some repeated columns, 
    \item for some subset $\mathcal{I}=\{i_0=0,i_1,\ldots,i_s\}$ of $\{0,1,2,\ldots,n\}$ and some ordering of the vertices $\sum_{h=0}^s A_{i_h}=I_w\otimes J_r$ for some integers $w$ and $r$ with $|X|=wr$, $1<w,r<|X|$,  
    \item for some subset $\mathcal{J}=\{j_0=0,j_1,\ldots,j_s\}$ of $\{0,1,2,\ldots,n\}$ and some ordering of the vertices $\sum_{h=0}^s E_{j_h}=\frac{1}{r}I_w\otimes J_r$ for some integers $w$ and $r$ with $|X|=wr$, $1<w,r<|X|$,  
\end{enumerate}
\end{theorem} 

A symmetric association scheme $(X,\{R_i\}_{i=0}^n)$ is said to be 
{\em $Q$-polynomial} 
if for each $i\in\{0,1,\ldots,n\}$, there exists a polynomial $v_i^*(x)$ of degree $i$ such that $q_{ji}=v_i^*(q_{j1})$ for all $j\in\{0,1,\ldots,n\}$. 
It is also known that a symmetric association scheme is a $Q$-polynomial if and only if the matrix $(q_{1,j}^k)_{j,k=0}^n$ is a tridiagonal matrix with nonzero superdiagonal and subdiagonals~\cite[p.193]{BI}. 
For a $Q$-polynomial association scheme, set $a_i^*=q_{1,i}^i$, $b_i^*=q_{1,i+1}^i$, and $c_i^*=q_{1,i-1}^i$. 
A $Q$-polynomial association scheme is {\em $Q$-antipodal} if $b_i^*=c_{n-i}^*$  except possibly for $i\neq \lfloor n/2\rfloor$. 
A $Q$-polynomial association scheme is {\em $Q$-bipartite} if $a_i^*=0$ for any $i\in\{1,\ldots,n\}$.
Then the following theorem has been shown in \cite{}. 
\begin{theorem}
Let $(X,\{R_i\}_{i=0}^n)$ be an imprimitive $Q$-polynomial association scheme. 
Then one of the following holds:
\begin{itemize}
\item $(X,\{R_i\}_{i=0}^n)$ is $Q$-antipodal and $\mathcal{J}=\{0,2,4,\ldots\}$;
\item $(X,\{R_i\}_{i=0}^n)$ is $Q$-bipartite and $\mathcal{J}=\{0,n\}$. 
\end{itemize}
\end{theorem}
In the case of $Q$-bipartite, $r=2$ and $w=|X|/2$ hold. 
In the case of $Q$-antipodal, it is important to find an upper bound for $w$. 

Next we explain how we obtain complex spherical codes from commutative association schemes. 
Let $(X,\{R_i\}_{i=0}^n)$ be a commutative association schemes. 
Let $E_i$ be a primitive idempotent such that $E_i$ has no repeated rows. 
Since $E_i$ is a Hermitian positive semidiefinite matrix, we may regard $G=\frac{|X|}{m_i}E_i$ as a Gram matrix of a set $\tilde{X}$ of points in $S^{m_i-1}$ with inner product set $A(\tilde{X})=\{\frac{q_{ji}}{m_i} \mid 1 \leq j\leq n\}$.  
Note that the size of $A(\tilde{X})$ is less than $n$ when $q_{j_1i}=q_{j_2i}$ for some distinct $j_1,j_2\in\{1,\ldots,n\}$. 
A summary is as follows. 
\begin{theorem}\label{thm:emd}
Let $(X,\{R_i\}_{i=0}^n)$ be a commutative association scheme. 
Assume that $E_i$ is a primitive idempotent such that $E_i$ has no repeated rows. 
Then there exists  a set $\tilde{X}$ of points in $S^{m_i-1}$ with inner product set $A(\tilde{X})=\{\frac{q_{ji}}{m_i} \mid 1 \leq j\leq n\}$, which is distance invariant with valencies $k_i$.  
\end{theorem}

\section{Semidefinite programming for $\mathbb{C}^n$}\label{sec:SDP}
\subsection{More on polynomials of the unit sphere}\label{sec:polynomials}
Fix a vector $e$ in $\Om(d)$ and consider the stabilizer $H$ of $e$ in unitary group $U(d)$.
Obviously $H$ is isomorphic to $U(d-1)$.
The subgroup $H$ also acts on $\Harm_d(k,l)$.   
The irreducible decomposition of $\Harm_d(k,l)$ is as follows;
\begin{align*}
\Harm_d(k,l)=\bigoplus_{i=0}^k \bigoplus_{j=0}^l H_{d-1}^{k,l}(i,j),
\end{align*}
where $H_{d-1}^{k,l}(i,j)$ is isomorphic to $\Harm_{d-1}(i,j)$.

\subsection{Zonal matrices associated to isotypic components}
The decomposition for $\Harm_d(k,l)$ implies that of $\Hom_d(k,l)$ as follows;
\begin{align*}
\Hom_d(k,l)=\bigoplus_{i=0}^k \bigoplus_{j=0}^l m_{k,l}^{i,j}\Harm_{d-1}(i,j),
\end{align*}
where $m_{k,l}^{i,j}=\min(k-i,l-j)+1$.
Let $m=m_{k,l}^{i,j}$ for the ease of notation. 
For an isotropic component $I_{k,l}(i,j):=m_{k,l}^{i,j}\Harm_{d-1}(i,j)$, 
we consider the zonal matrix.
Let $S_a$ be $H_{d-1}^{k-m+a,l-m+a}(i,j)$ for $0\leq a\leq m$.
Fix an orthonormal basis $\{e_{0,1},\ldots,e_{0,h}\}$ in $S_0$.
Let $\phi_i$ be an isomorphism from $S_0$ to $S_i$ and 
$\{e_{i,1},\ldots,e_{i,h}\}$ be the image under $\phi_i$ of the above fixed basis.
For $x\in \Omega(d)$, define $E_{k,l}^{i,j}(x)$ to be the matrix $(e_{i,j}(x))_{\substack{1 \leq i \leq m\\ 1\leq j \leq h}}$.
Then we define the zonal matrix $Z_{k,l}^{i,j}(x,y)=E_{k,l}^{i,j}(x)(E_{k,l}^{i,j}(y))^*$ for $x,y\in \Omega(d)$.
The zonal matrix $Z_{k,l}^{i,j}(x,y)$ satisfies the following positive semidefiniteness property.
\begin{lemma}
Let $X$ be a finite set in $\Omega(d)$.
Then $\sum_{x,y\in X}Z_{k,l}^{i,j}(x,y)\succeq 0$. 
\end{lemma}
\begin{proof}
\begin{align*}
\sum_{x,y\in X}Z_{k,l}^{i,j}(x,y)=\sum_{x,y\in X}E_{k,l}^{i,j}(x)(E_{k,l}^{i,j}(y))^*=(\sum_{x\in X}E_{k,l}^{i,j}(x))(\sum_{y\in X}E_{k,l}^{i,j}(y))^*\succeq0.
\end{align*}
\end{proof}
Since $Z_{k,l}^{i,j}(gx,gy)=Z_{k,l}^{i,j}(x,y)$ for all $g\in U(d-1)$, 
we may write $Z_{k,l}^{i,j}(x,y)=Y_{k,l}^{i,j}(e_d^*x,e_d^*y,x^*y)$.
So we determine $Y_{k,l}^{i,j}(u,v,t)$.

For $z\in \Omega(d)$, we denote $z=u e+\sqrt{1-|u|^2}z'$ where $u\in \mathbb{C}, z'\in \Omega(d-1)$.
Define $\varphi:\Hom_{d-1}(k,l)\rightarrow \Hom_{d}(k,l)$ by $\varphi(f)(z)=(1-|u|^2)^{(k+l)/2}f(z')$.
Then 
\begin{itemize}
\item $\{\varphi(f)P(u):f\in\Harm_{d-1}(k,l),P\in u^{k-l}\mathbb{R}_{\leq m}[u\bar{u}]\}$ if $k\geq l$, and
\item $\{\varphi(f)P(u):f\in\Harm_{d-1}(k,l),P\in u^{-k+l}\mathbb{R}_{\leq m}[u\bar{u}]\}$ if $k\leq l$
\end{itemize}
are isomorphic to $m+1$ copies of $\Harm_{d-1}(k,l)$, where 
\begin{align*}
\mathbb{R}_{\leq m}[u\bar{u}]=\{f\mid  f(u)=\sum_{i=0}^m a_i(u\bar{u})^i, a_i\in\mathbb{R} \}. 
\end{align*}

We find a bivariate polynomial $P_{a,a}(u)$ of degree $a$ with respect to $u \bar{u}$ such that $\varphi(\Harm_{d-1}(i,j))P_{a,a}(u)=H_{d-1}^{k-m+a,l-m+a}(i,j)$. 
For $u=x+\sqrt{-1}y$ with real numbers $x,y$, we also write $P_{a,a}(u)=P_{a,a}(x,y)$. 

The measures $w$, $w'$ on $\Omega(d)$, $\Omega(d-1)$ are related as follows:
\begin{align*}
d w(z)=(1-x^2-y^2)^{d-2}d x d y d w'(z'),
\end{align*}
where $u=x+\sqrt{-1}y$.
We determine first $P_{a,a}(u)$.
Let $\mathbb{D}=\{x\in \mathbb{C}\mid  |x|\leq 1\}$. 
For distinct pairs $(a,a),(b,b)$ and $f\in \Harm_{d-1}(i,j)$,  
\begin{align*}
0&=\int_{z\in\Omega(d)}\varphi(f)P_{a,a}(u)\overline{\varphi(f)}\overline{P_{b,b}(u)}dw(z) \\
&=\int_{z\in\Omega(d)}|f(z')|^2 P_{a,a}(x,y)\overline{P_{b,b}(x,y)}(1-x^2-y^2)^{i+j}dw(z) \\
&=\int_{z'\in\Omega(d-1)}|f(z')|^2d w'(z') \cdot \int_{(x,y)\in\mathbb{D}}P_{a,a}(x,y)\overline{P_{b,b}(x,y)}(1-x^2-y^2)^{i+j+(d-2)/2}d x d y, 
\end{align*}
which implies that  $\int_{(x,y)\in\mathbb{D}}P_{a,a}(x,y)\overline{P_{b,b}(x,y)}(1-x^2-y^2)^{i+j+(d-2)/2}d x d y=0$.
Thus $P_{a,a}(x,y)$ is a multiple of $P_{a,a}^{i+j+(d-2)/2}(x,y)$, where $P_{a,a}^{\lambda}(x,y)$ is the orthogonal polynomials with respect to the weight function $(1-x^2-y^2)^{\lambda}$.

Let $\{f_1,\ldots,f_h\}$ be an orthonormal basis in $\Harm_{d-1}(i,j)$.
Let $\lambda_{a}$ be such that $e_{a,l}=\lambda_{p}\varphi(f_l)P_{a,a}^{i+j+(d-2)/2}(x,y)$ ($1\leq l\leq h$) form an orthonoramal basis in $H_{d-1}^{k-m+a,l-m+a}(i,j)$.
Then 
$$|\lambda_{a}|^2=\frac{w_{d}}{w_{d-1}}\int_{\mathbb{D}}|P_{a,a}^{i+j+(d-2)/2}(x,y)|^2(1-x^2-y^2)^{i+j+(d-2)/2}d x d y.$$
Then for $1\leq a,b\leq m_{k,l}^{i,j} $ and $z_l=u_l e+ \sqrt{1-|u_l|^2}z_l'$ ($l=1,2$)
\begin{align*}
&(Z_{k,l}^{i,j}(z_1,z_2))_{a,b}\\
&=
\sum_{s=1}^h e_{a,s}(z_1)\overline{e_{b,s}(z_2)}\\
&=
\sum_{s=1}^h \lambda_a f_s(z_1')(1-|u_1|^2)^{(i+j)/2} P_{a,a}^{i+j+(d-2)/2}(u_1) \overline{\lambda_b} \overline{f_s(z_2')}(1-|u_2|^2)^{(i+j)/2} \overline{P_{b,b}^{i+j+(d-2)/2}(u_2)}  \\
&=\lambda_a\overline{\lambda_b}(1-|u_1|^2)^{(i+j)/2}(1-|u_2|^2)^{(i+j)/2}   P_{a,a}^{i+j+(d-2)/2}(u_1) \overline{P_{b,b}^{i+j+(d-2)/2}(u_2)} 
\sum_{s=1}^hf_s(z_1')\overline{f_s(z_2')} \\
&=\lambda_a\overline{\lambda_b}(1-|u_1|^2)^{(i+j)/2}(1-|u_2|^2)^{(i+j)/2}  P_{a,a}^{i+j+(d-2)/2}(u_1) \overline{P_{b,b}^{i+j+(d-2)/2}(u_2)} g_{i,j}^{d-1}(z_1'^* z_2')
\end{align*}
Thus we obtain the formula for $Y_{k,l}^{i,j}(u,v,t)$. 
\begin{theorem}
For $0\leq a,b\leq m_{k,l}^{i,j}$, 
\begin{align*}
&(Y_{k,l}^{i,j}(u,v,t))_{a,b}\\
&=\lambda_a\overline{\lambda_b}(1-u\bar{u})^{(i+j)/2}(1-v\bar{v})^{(i+j)/2}   P_{a,a}^{i+j+(d-2)/2}(u) \overline{P_{b,b}^{i+j+(d-2)/2}(v)} g_{i,j}^{d-1}(\frac{t-\bar{u} v}{\sqrt{(1-u\bar{u})(1-v\bar{v})}}). 
\end{align*}
\end{theorem}

Replacing $Y_{k,l}^{i,j}$ with $UY_{k,l}^{i,j} U^\top$ for an invertible matrix $U$ preserves the positive semidefiniteness.   
So we may take $Y_{k,l}^{i,j}$ as 
\begin{align*}
(Y_{k,l}^{i,j}(u,v,t))_{a,b}=(1-u\bar{u})^{(i+j)/2}(1-v\bar{v})^{(i+j)/2}(u\bar{u})^a(v\bar{v})^b g_{i,j}^{d-1}(\frac{t-\bar{u}v}{\sqrt{(1-u\bar{u})(1-v\bar{v})}}).
\end{align*}

\subsection{Semidefinite programming on complex unit sphere}
Let $X$ be a nonempty finite subset in $\Omega(d)$ with the inner product set $A(X)=\{x^*y\mid  x,y\in X,x \neq y\}$.
Let $X$ be a non-empty finite set of $\Omega(d)$, and $e\in \Omega(d)$ be an arbitrary point. 
Then 
\begin{align*} 
\sum_{x,y\in X}Y_{k,l}^{i,j}(e^*x,e^*y,x^*y)\succeq 0. 
\end{align*}
Thus we have
\begin{align*} 
\sum_{x,y,z\in X}Y_{k,l}^{i,j}(x^*y,x^*z,y^*z)\succeq 0. 
\end{align*}
We define three-points distance distribution 
\begin{align*}
x(u,v,t):=\frac{1}{|X|}|\{(x,y,z)\in X^3\mid  x^*y=u,x^*z=v,y^*z=t \}|
\end{align*}
for $u,v,t \in \mathbb{D}$ such that the matrix 
\begin{align*}
\begin{pmatrix}
1 & u & v \\
\bar{u} & 1 & t \\
\bar{v} & \bar{t} & 1
\end{pmatrix}
\end{align*}
is positive semidefinite. 
It is easy to see that 
\begin{align*}
x(u,u,1)&=x(\bar{u},\bar{u},1),\quad x(u,1,\bar{u})=x(\bar{u},1,u),\quad x(1,u,u)=x(1,\bar{u},\bar{u}), \\
x(u,v,t)&=x(v,u,\bar{t})=x(\bar{t},\bar{v},\bar{u})=x(\bar{v},\bar{t},u)=x(\bar{u},t,v)=x(t,\bar{u},\bar{v}). 
\end{align*}

We list properties on $x(u,v,t)$:
\begin{align*}
x(u,v,t)&\geq0,\\
x(1,1,1)&=1,\\
\sum_{u,v,t}x(u,v,t)&=|X|^2,\\
\sum_{u}x(u,u,1)&=|X|, \\
\sum_{u}x(u,u,1)g_{k,l}^{d}(u)&\geq 0 \text{ for } k,l,\\
\sum_{u,v,t}x(u,v,t)Y_{k,l}^{i,j}(u,v,t)&\succeq 0 \text{ for } k,l,i,j. 
\end{align*}

\begin{theorem}\label{thm:sdpbound}
Let $A$ be a subset of $\mathbb{D}\setminus \{1\}$ 
such that $\{\bar{c}\mid  c\in A\}=A$
and 
\begin{align*}
D&=\{(u,v,t)\in A^3\mid  \left(\begin{smallmatrix}
1 & u & v \\
\bar{u} & 1 & t \\
\bar{v} & \bar{t} & 1
\end{smallmatrix}\right)\succeq 0 \}, \\
D_0&=\{(u,u,1)\mid u\in A\}\cup \{(u,1,\bar{u})\mid u\in A\} \cup \{(1,u,u)\mid u\in A\}. 
\end{align*}   
Let $X$ be a complex spherical code in $\Omega(d)$ such that $A(X)\subset A$. 
Then the following semidefinite program is an upper bound for $|X|$: 
\begin{align*}
1+&\max \Big\{\sum_{u\in A}x(u,u,1)\mid  \\
&x(u,v,t)=x(v,u,\bar{t})=x(\bar{t},\bar{v},\bar{u})=x(\bar{v},\bar{t},u)=x(\bar{u},t,v)=x(t,\bar{u},\bar{v}) \text{ for }u,v,t\in A, \\
&x(u,v,t)\geq 0 \text{ for }(u,v,t)\in D,\\
&x(u,v,t)=0 \text{ for all but finitely many }(u,v,t)\in D,\\
& \left(\begin{smallmatrix}
1 & 0 \\
0 & 0
\end{smallmatrix}\right)
+ \sum_{u\in A}x(u,u,1)\left(\begin{smallmatrix}
0 & 1 \\
1 & 1
\end{smallmatrix}\right)
+ \sum_{(u,v,t)\in D}x(u,v,t)\left(\begin{smallmatrix}
0 & 0 \\
0 & 1
\end{smallmatrix}\right) \succeq 0, \\
& \textbf{1} 
g_{k,l}^d(1)+\sum_{u\in A}x(u,u,1)g_{k,l}^{d}(u)\geq 0 \text{ for } k,l,\\
& \textbf{1} 
Y_{k,l}^{i,j}(1,1,1)+\sum_{(u,v,t)\in D\cup D_0}x(u,v,t)Y_{k,l}^{i,j}(u,v,t)\succeq 0 \text{ for } k,l,i,j 
\Big\}.
\end{align*}
\end{theorem}

\section{Complex $2$-codes}\label{sec:2code}
Let $X$ be a complex $2$-code in $\Omega(d)$ with inner product set $A(X)=\{\alpha,\bar{\alpha}\}$. 
Define a $(0,1)$-matrix $A$ with rows and columns indexed by the elements of $X$ such that $(x,y)$-entry equals to $1$ if and only if $x^* y=\alpha$. 
The matrix $A$ is the adjacency matrix of a tournament, that is a oriented complete graph and satisfies that $A+A^\top=J-I$.
Then the following holds. 
\begin{theorem}{\rm \cite[Theorem~4.1,Theorem~4.7, Theorem~4.8]{NS2016}}\label{thm:2code}
Let $X$ be a complex $2$-code in $\Omega(d)$ with inner product set $A(X)=\{\alpha,\bar{\alpha}\}$. 
\begin{enumerate}
    \item $|X|\leq 2d+1$ holds. $|X|=2d+1$ if and only if $d$ is odd and $AA^\top=(d+1)I+dJ$. (The matrix $A$ is the adjacency matrix of a {\it doubly regular tournament}.)  In this case, $\alpha=\frac{-1+\sqrt{-1-2d}}{2d}$.
    \item If $d$ is even, then $|X|\leq 2d$. $|X|=2d$ holds if and only if $I+A-A^\top$ is a skew Hadamard matrix of order $2d$. In this case, $\alpha=\frac{\sqrt{-1}}{\sqrt{2d-1}}$. 
    \item $d$ is odd and $|X|=2d$ if and only if one of the following holds:
    \begin{enumerate}
        \item $X$ is obtained from some complex spherical $2$-code $Y$ in $\Omega(d)$ with $|Y|=2d+1$ by deleting a point.
        \item $A$ satisfies that 
        \begin{align*}
            (A-A^\top)(A^\top-A)=\begin{pmatrix}
            kI+(2d-1-k)J&O\\
        O & kI+(2d-1-k)J
            \end{pmatrix}
        \end{align*}
        for some odd integer $k$ such that $1\leq k\leq 2d-3$. In this case, $\alpha=\frac{2d-k-1+2\sqrt{-k}}{2d+k-1}$.
    \end{enumerate}
\end{enumerate}
\end{theorem}

In the following, we explain how to obtain the complex $2$-codes from a matrix $A$ in Theorem~\ref{thm:2code} 3.(b).   
Let  $G=(V,E)$ be a tournament of order $n$ with adjacency matrix $A$. 
Since $\sqrt{-1}(A-A^T)$ is a Hermitian matrix,  
the eigenvalues of $\sqrt{-1}(A-A^T)$ are real. 
Let $\tau_1<\cdots< \tau_s$ be the eigenvalues of $\sqrt{-1}(A-A^T)$.  
Let $E_i$ be the orthogonal projection matrix onto the eigenspace
corresponding to $\tau_i$ ($i=1,2,\ldots,s$). 
The {main angle} $\beta_i$ of $\tau_i$ is defined to be the value 
$
\beta_i= 
\frac{1}{\sqrt{n}} \sqrt{(E_i j) ^\ast (E_i j)},
$
where 
$j$ is the all-ones vector (see~\cite[Section 2]{NS2016}).
A {representation} in $\Omega(d)$ of $G$
is a mapping $\varphi$ from $V$ to $\Omega(d)$ such that there exists 
an imaginary number $\gamma$ with 
$\operatorname{Im}(\gamma)>0$ satisfying the property that 
for all distinct $x,y \in V$, 
$\varphi(x)^* \varphi(y)=\gamma$  if $(x,y) \in E$ and $\varphi(x)^* \varphi(y)=\overline{\gamma}$ if $(y,x) \in E$.
Then the Gram matrix of the image of the representation $\varphi$ of $G$ is given by:
\begin{align}\label{eq:gram}
I+\gamma A+\overline{\gamma} A^T.
\end{align} 
A representation $\varphi$ in $\Omega(d)$ of $G$ is {minimal} if $d$ 
is the smallest $e$ such that there exists a representation
in $\Omega(e)$ of $G$.

Let $X$ be a complex spherical $2$-code in $\Omega(d)$ with $|X|=2d$ for $d$ odd,
corresponding to Theorem~\ref{thm:2code} 3.(b) and $A$ the adjacency matrix of the tournament attached to $X$. 
Then the Gram matrix of $X$ is the same as the Gram matrix of the image of
the minimal representation of the tournament~\cite[Theorem~4.8 (ii)]{NS2016},
and the Gram matrix of $X$ is given by \eqref{eq:gram} with $\gamma=(1-c_2\sqrt{-1})/(1+c_2\tau_2)$, where $c_2=2d(\beta_1^2/(\tau_1-\tau_2)+\sum_{i=3}^s\beta_i^2/(\tau_i-\tau_2))$ \cite[Theorem~3.1 (3)]{NS2016}. 
After the suitable permuting the rows and the columns simultaneously the matrix $(\sqrt{-1}(A-A^T))^2$ is 
\begin{align*}
\begin{pmatrix}
(\alpha-1) I+\beta J& O\\
O&(\alpha-1) I+\beta J
\end{pmatrix},
\end{align*} 
for some positive integers $\alpha$ and $\beta$,
noting that $(\sqrt{-1}(A-A^T))^2 = (I+A-A^T)(I+A-A^T)^T - I$.  
Thus, $\sqrt{-1}(A-A^T)$ has eigenvalues  $\tau_1=-\sqrt{\alpha+\beta d-1},\tau_2=-\sqrt{\alpha-1},\tau_3=\sqrt{\alpha-1},\tau_4=\sqrt{\alpha+\beta d-1}$ with 
multiplicities $1,d-1,d-1,1$, respectively. 
The main angles are $\beta_1=\beta_4=1/\sqrt{2}, \beta_2=\beta_3=0$ by 
the proof of Lemma~4.5 in~\cite{NS2016}.
Then $c_2=-2\sqrt{\alpha-1}/\beta$.
Therefore, the Gram matrix of $X$ is determined by using $\alpha,\beta$
as follows:
\begin{align*}
I+\gamma A+\overline{\gamma} A^T&=I+\frac{-c_2\sqrt{-1}}{1+c_2\tau_2}(A-A^T)+\frac{1}{1+c_2\tau_2}(A+A^T)\\
&=I+\frac{2\sqrt{\alpha-1}\sqrt{-1}}{2\alpha+\beta-2}(A-A^T)+\frac{\beta}{2\alpha+\beta-2}(J-I).
\end{align*}

For the case $k=2d-3$, the matrix $I+A-A^\top$ of order $2d$ is known as a {\it D-optimal design}, namely it has the largest determinant among all $2d\times 2d$ $(1,-1)$-matrices where $d$ is odd. 
One of the constructions of Theorem~\ref{thm:2code}3.(b) comes from {\it skew-symmetric supplementary difference sets}.
For $X \subset \mathbb{Z}_{d}$ and $i \in \mathbb{Z}_{d}$, define
\begin{align*}
P_{X}(i) &= |\{(x,y) \in X \times X \mid y-x = i\}| \text{ and } \\
P_{X} &= (P_{X}(1),P_{X}(2),\ldots,P_{X}(d-1)).
\end{align*}
Let $X_1$ and $X_2$ be an $k_1$-subset and a
$k_2$-subset of $\mathbb{Z}_d$, respectively.
If a pair $(X_1,X_2)$ satisfies 
$$P_{X_1}+P_{X_2}=(\lambda,\lambda,\ldots,\lambda),$$
then it is called a {\em $2$-$\{d;k_1,k_2;\lambda\}$ supplementary difference
set}.
Let $R_{1}$ and $R_{2}$ be the circulant $d \times d$ $(1,-1)$-matrices with
first rows $r_1=(r_{1,1},r_{1,2},\ldots,r_{1,d})$ and 
$r_2=(r_{2,1},r_{2,2},\ldots,r_{2,d})$, respectively, 
defined as follows:
$r_{1,i+1}=-1$ if $i \in X_1$,
$r_{1,i+1}=1$ if $i \not\in X_1$
and 
$r_{2,i+1}=-1$ if $i \in X_2$,
$r_{2,i+1}=1$ if $i \not\in X_2$.
Then it is easy to see that
$$R_1R_1^\top+R_2R_2^\top=4(k_1+k_2-\lambda)I+2(d-2(k_1+k_2-\lambda))J.$$

A supplementary difference set $(X_1,X_2)$ is called {\it skew-symmetric} if the matrix $R_1$ satisfied that $R_1+R_1^\top=2I$. 
For a skew-symmetric supplementary difference set $(X_1,X_2)$, define $S=\begin{pmatrix}
R_{1}-I & R_{2} \\
R_{2}^\top & R_{1}^\top-I
\end{pmatrix}$. 
Then $S=A-A^\top$ for some $(0,1)$-matrix $A$ satisfying $A+A^\top=J-I$ and 
\begin{align*}
    (A-A^\top)(A^\top-A)&=\begin{pmatrix}
    kI+(2d-1-k)J&O\\
        O & kI+(2d-1-k)J
    \end{pmatrix},
\end{align*}
where $k=4(k_1+k_2-\lambda)-1$. 
Therefore, a $2$-$\{d;k_1,k_2;\lambda\}$ skew-symmetric supplementary difference set with parameters gives a complex spherical $2$-codes in $\Omega(d)$ with  $$A(X)=\{\frac{d-2(k_1+k_2-\lambda)\pm\sqrt{-(4(k_1+k_2-\lambda)-1)}}{d+2(k_1+k_2-\lambda)-1}\}.$$ 

Skew-symmetric supplementary difference sets are studied  for $d\leq 75$ in \cite{AHS16}. 
There are 9 sets of parameters $(d,k_1,k_2,\lambda)$ such that skew-symmetric supplementary difference sets are not known to exist: 
\begin{align*}
    (d,k_1,k_2,\lambda)\in& \{(53,26,14,16),(57,28,21,21),(61,30,15,18),(67,33,12,18),\\&(67,33,22,23),(69,34,18,21),(71,35,15,20),(71,35,21,23),(73,36,28,28)\}.
\end{align*}

\begin{problem}
\begin{enumerate}
    \item For which $d,k$ does there exist a complex spherical $2$-code $X$ in $\Omega(d)$ with $|X|=2d$ and  $A(X)=\{\frac{2d-k-1\pm 2\sqrt{-k}}{2d+k-1}\}$? 
    In other words, what is the upper bound for complex spherical $2$-codes in $\Omega(d)$ with $A(X)=\{\frac{2d-k-1\pm 2\sqrt{-k}}{2d+k-1}\}$?
    \item For which $(d,k_1,k_2,\lambda)$ such that 
    \begin{align*}
    (d,k_1,k_2,\lambda)\in& \{(53,26,14,16),(57,28,21,21),(61,30,15,18),(67,33,12,18),\\&(67,33,22,23),(69,34,18,21),(71,35,15,20),(71,35,21,23),(73,36,28,28)\}, 
\end{align*}
    does there exist a complex spherical $2$-code $X$ in $\Omega(d)$ with $|X|=2d$ and  $$A(X)=\left\{\frac{d-2(k_1+k_2-\lambda)\pm\sqrt{-4(k_1+k_2-\lambda)+1}}{d+2(k_1+k_2-\lambda)-1}\right\}?$$
\end{enumerate}
\end{problem}

The following theorem answers Problem 4.2. for $k=1$ and $d \geq 5$. The proof is a direct consequence of Theorem~\ref{thm:sdpbound} with some calculations.

\begin{theorem}
Let $X$ be a compelx $2$-code in $\Omega(d)$ with $A(X)=\{\frac{d-1\pm \sqrt{-1}}{d}\}$. If $d\geq 5$, then $|X|<6$. 
\end{theorem}
\begin{proof}
Let $\alpha = \frac{d-1+\sqrt{-1}}{d}$. The distance distributions for $X$ with $A(X) = \{\alpha, \bar{\alpha}\}$ are
\begin{align*}
    x_1 :&= x(\alpha, \alpha, 1) = x(\bar{\alpha}, \bar{\alpha}, 1), \\
    x_2 :&= x(\alpha, \alpha, \alpha) = x(\alpha, \alpha, \bar{\alpha}) = x(\alpha, \bar{\alpha}, \bar{\alpha}) \\
    &= x(\bar{\alpha}, \bar{\alpha}, \bar{\alpha}) = x(\bar{\alpha}, \bar{\alpha}, \alpha) = x(\bar{\alpha}, \alpha, \alpha), \\
    x_3 :&= x(\alpha, \bar{\alpha}, \alpha) = x(\bar{\alpha}, \alpha, \bar{\alpha}),
\end{align*}
and $|X| = 1 + 2x_1$.

Consider $(Y^{3,0}_{k,l}(u,v,t))_{0,0} = \frac{(d-1)d(d+1)}{6}(t-\bar{u}v)^3$. By Theorem~\ref{thm:sdpbound}, we have
\begin{align*}
    & (Y_{k,l}^{3,0}(1,1,1))_{0,0}+\sum_{(u,v,t)\in D\cup D_0}x(u,v,t)(Y_{3,0}^{i,j}(u,v,t))_{0,0} \\
    =& x_1 \cdot 2(1-\alpha\bar{\alpha})^3 + x_2 \Big(2(\alpha-\alpha\bar{\alpha})^3 + 2(\bar{\alpha}-\alpha\bar{\alpha})^3 + (\alpha-\alpha^2)^3 + (\bar{\alpha}-\bar{\alpha}^2)^3\Big) + x_3\Big((\alpha-\bar{\alpha}^2)^3 + (\bar{\alpha}-\alpha^2)^3\Big) \geq 0.
\end{align*}

Plug in $\alpha = \frac{d-1+\sqrt{-1}}{d}$, we have $Ax_1 - Bx_2 - Cx_3 \geq 0$, where
\begin{align*}
    A &= 4(d-1)^3, \\
    B &= 3d^3-6d^2-6d+8, \\
    C &= 13d^3 - 18d^2 + 6d.
\end{align*}
We rewrite the inequality as
\[
    Ax_1 - (C - \frac{B}{3})x_3 \geq \frac{B}{6}(6x_2+2x_3).
\]
for $d \geq 5$, both $B$ and $C - B/3$ are positive. Therefore we have $6x_2 + 2x_3 \leq \frac{6A}{B}x_1$.
On the other hand, the $2 \times 2$ counting matrix in Theorem~\ref{thm:sdpbound} is $\begin{pmatrix} 1 & 2x_1 \\ 2x_1 & 2x_1+6x_2+2x_3 \end{pmatrix} \succeq 0$. Therefore
\[
    (2x_1)^2 \leq 2x_1+6x_2+2x_3 \leq \frac{6A+2B}{B}x_1 \Longrightarrow x_1 \leq \frac{3A+B}{2B}.
\]
So
\[
    |X| = 1+2x_1 \leq \frac{3A+2B}{B} = \frac{18d^3-48d^2+24d+4}{3d^3-6d^2-6d+8},
\]
which is less than $6$ when $d \geq 5$.
\end{proof}

\section{Complex $3$-codes attached to non-symmetric schemes with $4$ classes}\label{sec:complex3code}
In \cite{MW}, skew-symmetric association schemes with $4$ classes were studied. 
A commutative association scheme is said to be \defn{skew-symmetric} if there is no symmetric binary relation other than the diagonal relation. 
We focus on $4$-classes case here. 
Let $\mathfrak{X}=(X,\{R_0,R_1,R_1^\top,R_2,R_2^\top\})$ be a skew-symmetric association scheme. 
It is shown in \cite[Theorem~2.1]{MW} that the the first eigenmatrix is described as follows. 
It is known that $\tilde{\mathfrak{X}}=(X,\{R_0,R_1\cup R_1^\top,R_2\cup R_2^\top\})$ is a symmetric association scheme. 
Let 
\begin{align*}
\tilde{P}=
\begin{pmatrix}
1 & k_1 & k_2 \\
1 & r & t \\
1 & s & u 
\end{pmatrix}    
\end{align*}
be the first eigenmatrix of $\tilde{\mathfrak{X}}$ and $m_i$ the multiplicit corresponding to $i$-row. 
Then the first eigenmatrix $P$ of $\mathfrak{X}$ is 
\begin{align*}
P=
\begin{pmatrix}
1 & k_1/2 & k_1/2 & k_2/2 & k_2/2  \\
1 & \rho & \tau & \overline{\tau} & \overline{\rho}  \\
1 & \sigma & \omega & \overline{\omega} & \overline{\sigma}  \\
1 & \overline{\sigma} & \overline{\omega} & \omega & \sigma  \\
1 & \overline{\rho} & \overline{\tau} & \tau & \rho  
\end{pmatrix}    
\end{align*}
and the corresponding multiplicities are $1,m_1/2,m_2/2,m_2/2,m_1/2$. 
The entries $\rho,\omega,\tau$, and $\sigma$ are one of the the three cases: 
\begin{enumerate}[(i)]
    \item $\rho=r/2,\sigma=(s+\sqrt{-b})/2,\tau=(t+\sqrt{-z})/2,\omega=u/2$, where $b=|X|k_1/m_2, z=|X|k_2/m_1$. 
    \item $\rho=(r+\sqrt{-y})/2,\sigma=s/2,\tau=t/2,\omega=(u+\sqrt{-c})/2$, where $y=|X|k_1/m_1,c=|X|k_2/m_2$. 
    \item $\rho=(r+\sqrt{-y})/2,\tau=(t+\sqrt{-z})/2,\sigma=(s+\sqrt{-b})/2,\omega=(u-\sqrt{-c})/2$, where all of $b,c,y,z$ are positive and satisfy the following equations: 
    $$
    m_1y+m_2b=|X|k_1,\quad m_1z+m_2c=|X|k_2,\quad m_1\sqrt{yz}-m_2\sqrt{bc}=0.
    $$
\end{enumerate}
We denote by $P_I,P_{II},P_{III}$ the three first eigenmatrix, and refer them as type I, II, and III, respectively.  
In \cite[Table~2]{MW}, there were $15$ feasible parameters which are not known to exist. 
The following non-existence result for two of the above  is the main result in this subsection. 

\begin{theorem}\label{thm:as4skew}
There exist no skew-symmetric association schemes with $4$ classes with the following parameters:
\begin{enumerate}
    \item $(|X|,k,\lambda,\mu,r^{m_1},s^{m_2},\text{type})=(889,222,35,62,5^{762},-32^{126},\text{II})$
    \item $(|X|,k,\lambda,\mu,r^{m_1},s^{m_2},\text{type})=(945,354,153,120,39^{118},-21^{330},\text{II})$
\end{enumerate}
\end{theorem}
\begin{proof}
Assume to the contrary that there are such association schemes. 
From the above parameters and by Theorem~\ref{thm:emd}, we have the following $3$-codes $\tilde{X}$ in $\Omega(d)$: 
\begin{enumerate}
    \item  $d=63$, $A(\tilde{X})=\left\{\alpha_1=-\frac{16}{111},\alpha_2=\frac{31}{666}+\frac{1}{18} i \sqrt{\frac{127}{37}},\alpha_3=\frac{31}{666}-\frac{1}{18} i \sqrt{\frac{127}{37}}\right\}$,
    \item  $d=59$, $A(\tilde{X})=\left\{\alpha_1=-\frac{4}{59},\alpha_2=\frac{13}{118}+\frac{3 i \sqrt{35}}{118},\alpha_3=\frac{13}{118}-\frac{3 i \sqrt{35}}{118}\right\}$. 
\end{enumerate} 
By Theorem~\ref{thm:emd}, each $\tilde{X}$ is distance invariant with valencies $\tilde{k}_1,\tilde{k}_2,\tilde{k}_3$ corresponding to $\alpha_1,\alpha_2,\alpha_3$ given as 
\begin{enumerate}
    \item $(\tilde{k}_1,\tilde{k}_2,\tilde{k}_3)=(222,333,333)$, 
    \item $(\tilde{k}_1,\tilde{k}_2,\tilde{k}_3)=(590,177,177)$. 
\end{enumerate}
Recall that 
$g_{3,0}^d(x) = \frac{d(d+1)(d+2)}{6}x^3$.  
Then the values 
$$
\frac{1}{|X|}\sum_{a,b \in \tilde{X}} g^d_{3,0} (a^* b)=g^d_{3,0}(1)+\tilde{k}_1 g^d_{3,0}(\alpha_1)+\tilde{k}_2g^d_{3,0}(\alpha_2)+\tilde{k}_3g^d_{3,0}(\alpha_3)
$$
are $-\frac{941202080}{36963}$ for (1) and  $-\frac{2882250}{59}$ for (2), which contradict to Lemma~\ref{lem:sum}. 
\end{proof}

\section{SDP for $3$-distance set in real unit sphere}\label{sec:real3code}

In this section, we improve the known upper bounds for three distance sets related to Hadamard matrices and give another proof for known results in $Q$-antipodal $Q$-polynomial association schemes by the LP and the SDP method. 

\subsection{Quasi-unbiased Hadamard matrices}\label{sec:qh}
A Hadamard matrix of order $d$ is a $d\times d$ $(1,-1)$-matrix $H$ such that $HH^\top=d I$, and 
a weighing matrix of order $d$ and weight $k$ is a $d\times d$ $(0,1,-1)$-matrix $W$ such that $WW^\top=k I$. 
A weighing matrix of order $d$ and weight $k$ is a Hadamard matrix of order $n$.  
Two Hadamard matrices $H_1,H_2$ of order $d$ are said to be {\em unbiased} if 
$(1/\sqrt{d})H_1 H_2^\top$ is a Hadamard matrix. 
The set of normalized row vectors of mutually unbiased Hadamard matrices $H_1,\ldots,H_f$ whose norm equals to $1$ formrs a set of bi-angular lines set in $\mathbb{R}^d$ with inner products $\pm1/\sqrt{d},0$.

Two Hadamard matrices $H_1,H_2$ of order $d$ are said to be {\em quasi-unbiased for parameters $(d,l,a)$} if 
$(1/\sqrt{a})H_1 H_2^\top$ is a weighing matrix of weight $l$. 
Note that $l=d^2/a$. 
Hadamard matrices $H_1,\ldots,H_f$ of order $d$ are {\em mutually quasi-unbiased for parameters $(d,l,a)$} if any two distinct are quasi-unbiased for the parameters. 
The set of normalized row vectors of $H_1,\ldots,H_f$ whose norm equals to $1$ forms a set of bi-angular lines set in $\mathbb{R}^d$ with inner products $\pm1/\sqrt{l},0$.

Therefore, a set of $f$ quasi-unbiased Hadamard matrices for parameters $(d, l, a)$ gives a spherical three-distance set in $S^{d-1}$ with $|X| = df$ and $A(X) = \{\pm1/\sqrt{l},0\}$. The linear programming/semidefinite programming method (on spherical codes) improves the known upper bounds for $f$, the number of the Hadamard matrices for which they are quasi-unbiased as follows.

\begin{theorem}
Let $d,l$ be positive integers, and $f$ be the maximum number of Hadamard matrices of order $d$ for which the Hadamard matrices are quasi-unbiased for parameters $(d,l,a)$. Then the upper bound $f$ is given in the table:
\begin{table}[H]\centering
  \begin{tabular}{|c||c|c|c|c|} \hline
      $(d,l,a)$  & previous known bound on $f$~\cite{AHS} & Theorem~\ref{thm:Hadamard1} & Theorem~\ref{thm:Hadamard2} & SDP~\cite{Liu}  \\ \hline \hline 
   $(16,4,64)$  & $f\leq 35$ & & $f \leq 15$ & $f \leq 15$ \\ \hline 
   $(24,4,144)$ & $f\leq 85$ & & $f \leq 27$ & $f \leq 25$ \\ \hline
  $(24,9,64)$  & $f\leq 85$ & $f \leq 208$ & & $f \leq 95$ \\ \hline
  $(32,4,256)$ &  $f\leq 155$ & & $f \leq 63$ & $f \leq 47$ \\ \hline
  $(36,9,144)$ & $f\leq 199$ & & $f \leq 80$ & $f \leq 79$ \\ \hline
  $(40,4,400)$ & $f\leq 247$ & & & $f \leq 101$ \\ \hline
  $(40,25,64)$ & $f\leq 28$ & $f\leq 30$ & & $f \leq 30$ \\ \hline
  $(48,4,576)$ & $f\leq 361$ & & & $f \leq 276$ \\ \hline
 $(48,9,256)$ & $f\leq 361$ & & $f \leq 105$ & $f \leq 104$ \\ \hline
 $(48,16,144)$ & $f\leq 361$ & & & $f \leq 316$ \\ \hline
 $(48,36,64)$ & $f\leq 28$ & $f \leq 30$ & & $f \leq 30$ \\ \hline
  \end{tabular}
\end{table}
\end{theorem}

Note that every known upper bound on $f$ is derived from the absolute bound for self-complementary codes or the LP for the binary Hamming schemes. 


The bounds in the third and the fourth column from the table above are obtained from Theorem~\ref{thm:Hadamard1} and Theorem~\ref{thm:Hadamard2}, respectively. The blank entry means that the conditions in the theorems are not satisfied.

The bounds in the last column are obtained from semidefinite programming on spherical three distance sets. The formula is developed in~\cite{BV} and~\cite{Liu}. We apply Theorem 3.1 in~\cite{Liu} with $\{d_1, d_2, d_3\} = \{\pm1/\sqrt{l}, 0\} = A(X)$ and $p_{LP} = p_{SDP} = 5$ here. 

\begin{lemma}[\cite{DGS}, Theorem 4.3]
\label{lem:lp}
Suppoose $X$ is a spherical three-distance set in $S^{d-1}$ with $A(X) = \{\alpha, \beta, \gamma\}$. Suppose $F(x)$ is a polynomial with Gegenbauer expansion
\[
    F(x) = \sum_{k \geq 0} f_kG^d_k(x).
\]
If $f_0 > 0$, $f_k \geq 0$ for all $k$, and $F(\alpha), F(\beta), F(\gamma) \leq 0$, then $|X| \leq \frac{F(1)}{f_0}$.
\end{lemma}
The notation $G^d_k( \cdot)$ is the Gegenbauer polynomial for dimension $d$ and degree $k$. 

\begin{theorem}\label{thm:Hadamard1}
Suppose $X$ is a spherical three-distance set in $S^{d-1}$ with $A(X) = \{\pm\alpha, 0\}$. If $(d+4)\al^2 \leq 6$ and $(d+2)\al^2 < 3$, then
\[
    |X| \leq d(d+2)\cdot\frac{1-\al^2}{3-(d+2)\al^2}.
\]
\end{theorem}
\begin{proof}
Let $F(x) = x^2(x+\al)(x-\al)$. The Gegenbauer polynomial expansion of $F(x)$ is $F(x) = \sum_{k=0}^4 f_kG^d_k(x)$, where
\begin{align*}
    f_4 &= \frac{(d-1)(d+1)}{(d+2)(d+4)}, \\
    f_3 &= 0, \\
    f_2 &= \frac{d-1}{d(d+4)}\Big(6-(d+4)\al^2\Big), \\
    f_1 &= 0, \\
    f_0 &= \frac{3-(d+2)\al^2}{d(d+2)}.
\end{align*}
The proof is followed by Lemma~\ref{lem:lp}.
\end{proof}

\begin{theorem}\label{thm:Hadamard2}
Suppose $X$ is a spherical three-distance set in $S^{d-1}$ with $A(X) = \{\pm\alpha, 0\}$. If $(d+6)(d+8)\al^4 - 15(d+6)\al^2 + 45 \leq 0$ and $(d+4)(d+8)\al^4 - 15(d+4)\al^2 + 30 < 0$, then
\[
    |X| \leq d(d+4) \cdot \frac{(1-\al^2)(15-(d+8)(1+\al^2))}{(d+4)(d+8)\alpha^4 - 15(d+4)\alpha^2 + 30}.
\]
\end{theorem}
\begin{proof}
Let $F(x) = x^2(x+\al)(x-\al)\Big((d+8)(x^2+\al^2)-15\Big)$. The Gegenbauer polynomial expansion of $F(x)$ is $F(x) = \sum_{k=0}^6 f_kG^d_k(x)$, where
\begin{align*}
    f_6 &= \frac{(d-1)(d+1)(d+3)}{(d+4)(d+6)}, \\
    f_5 &= f_4 = f_3 = 0, \\
    f_2 &= -\frac{d-1}{d(d+6)}\Big((d+6)(d+8)\al^4 - 15(d+6)\al^2 + 45\Big), \\
    f_1 &= 0, \\
    f_0 &= -\frac{1}{d(d+4)}\Big((d+4)(d+8)\al^4 - 15(d+4)\al^2 + 30\Big).
\end{align*}
The proof is followed by Lemma~\ref{lem:lp}.
\end{proof}

\subsection{$Q$-antipodal $Q$-polynomial schemes with $3$ classes}\label{sec:q}

In $Q$-antipodal $Q$-polynomial schemes with $3$ classes, the number $w$ of the equivalence classes on the vertex set in Theorem~\ref{thm:imas} is of particular interest. 
In \cite[Page 414]{MMW}, it was shown that $w\leq \frac{m_1+2}{2}$. 
The LP method provides upper bounds for $3$-codes obtained from $Q$-antipodal $Q$-polynomial schemes with $3$ classes and thus those for $w$. The resulting bounds are the same of or slightly more than  \cite[Page 414]{MMW} as shown in the following table. 

The parameters of $Q$-antipodal $Q$-polynomial schemes with $3$ classes can be determined from the number $w$ and the parameters of a symmetric $(v,k,\lambda)$-design attached to the fibers $X_1\cup X_2$. The second eigenamatrix $Q$ is 
$$
Q=\begin{pmatrix}
1 & v-1 & (w-1)(v-1) & w-1 \\
1 & \sqrt{\frac{(v-1)(v-k)}{k}} & -\sqrt{\frac{(v-1)(v-k)}{k}} & -1 \\
1 & -1 & -w+1 & w-1 \\
1 & -\sqrt{\frac{(v-1)k}{v-k}} & \sqrt{\frac{(v-1)k}{v-k}} & -1 
\end{pmatrix}.
$$
In the following table, 
$$
(v,k,\lambda)\in\{(36,15,6),(45,33,24),(64,36,20),(96,76,60),(100,45,20),(144,78,42),(153,57,21)\}. 
$$

\begin{table}[H]\centering
  \begin{tabular}{|c|c|c||c|c|c|} \hline
    $|X|$ & $m_1$ & $A(X)$ & known bound on $w$ & Theorem~\ref{thm:Qscheme1} & Theorem~\ref{thm:Qscheme2} \\ \hline \hline
$36w$ & 35 & $\{1/5,-1/35,-1/7\}$ & $w\leq 18$ & & $w \leq 20$ \\ \hline
$45w$ & 44 & $\{1/11,-1/44,-1/4\}$ & $w\leq 7$ & $w\leq 7$ & \\ \hline
$64w$ & 63 & $\{1/9,-1/63,-1/7\}$ & $w\leq 32$ & $w\leq 32$ & \\ \hline
$96w$ & 95 & $\{1/19,-1/95,-1/5\}$ & $w\leq 7$ & $w\leq 7$ & \\ \hline
$100w$ & 99 & $\{1/9,-1/99,-1/11\}$ & $w\leq 50$ & & $w \leq 52$ \\ \hline
$144w$ & 143 & $\{1/13,-1/143,-1/11\}$ & $w\leq 72$ & $w\leq 72$ & \\ \hline
$153w$ & 152 & $\{2/19,-1/152,-1/16\}$ & $w\leq 77
$ & & $w \leq 99$ \\ \hline
  \end{tabular}
\end{table}

\begin{theorem}[\cite{DGS}, Example 4.7]\label{thm:Qscheme1}
Suppose $X$ is a spherical three-distance set in $S^{d-1}$ with $A(X) = \{\al, \be, \gamma\}$. If
\begin{align*}
    \al+\be+\gamma &\leq 0, \\
    \al\be+\be\gamma+\gamma\al &\geq \frac{-3}{d+2}, \\
    \al+\be+\gamma + d\al\be\gamma &< 0,
\end{align*}
then
\[
    |X| \leq \frac{d(1-\al)(1-\be)(1-\gamma)}{-(\al+\be+\gamma+d\al\be\gamma)}.
\]
\end{theorem}

\begin{theorem}\label{thm:Qscheme2}
Let $\al \in (-1, 0)$, $\be \in (0, 1)$ and $\al+\be > 0$. Suppose $X$ is a spherical three-distance set in $S^{d-1}$ with $A(X) = \{\al, \be, \al\be\}$ and further $d = -1/\al\be$. If
\begin{align*}
    (\al+\be+1)(1-2\al\be) &\leq 3, \\
    (\al^2+\al\be+\be^2+\al+\be-1)(-8\al^3\be^3+6\al^2\be^2+11\al\be) &\geq (4\al^3\be^3+3\al^2\be^2+5\al\be) + 9\al\be(\al+\be) + 3(\al^2+\be^2), \\
    (\al^2+2\al\be+\be^2+\al+\be+1)(4\al^2\be^2-8\al\be) &< -6\al\be(\al+\be+3) - 3(\al^2+\be^2),
\end{align*}
then
\[
    |X| \leq \frac{2(1-\al^2)(1-\be^2)(1-\al\be)(1-2\al\be)}{\al\be\Big((\al^2+2\al\be+\be^2+\al+\be+1)(4\al^2\be^2-8\al\be) + 6\al\be(\al+\be+3) + 3(\al^2+\be^2)\Big)}.
\]
\end{theorem}
\begin{proof}
Let $d = -1/\al\be$ and 
\[
    F(x) = (x-\al)(x-\be)(x-\al\be)\Big(((2\al\be-1)(\al+\be+1)+3)x + (2\al^2\be^2+2\al\be+3\al+3\be)\Big).
\]
The Gegenbauer polynomial expansion of $F(x)$ is $F(x) = \sum_{k=0}^4 f_kG^d_k(x)$, where
\begin{align*}
    f_4 &= \frac{(1+\al\be)(1-\al\be)((2\al\be-1)(\al+\be+1)+3)}{(1-2\al\be)(1-4\al\be)}, \\
    f_3 &= (1+\al)(1+\be)(1+\al\be)(\al+\be), \\
    f_2 &= \frac{1+\al\be}{1-4\al\be} \cdot \Big((\al^2+\al\be+\be^2+\al+\be-1)(-8\al^3\be^3+6\al^2\be^2+11\al\be) \\
    &- (4\al^3\be^3+3\al^2\be^2+5\al\be) - 9\al\be(\al+\be) - 3(\al^2+\be^2)\Big), \\
    f_1 &= 0, \\
    f_0 &= \frac{\al\be(1+\al\be)}{1-2\al\be} \cdot \Big((\al^2+2\al\be+\be^2+\al+\be+1)(4\al^2\be^2-8\al\be) + 6\al\be(\al+\be+3) + 3(\al^2+\be^2)\Big).
\end{align*}
The proof is followed by Lemma~\ref{lem:lp}.
\end{proof}

\section*{Acknowledgements}
The authors would like to thank Aidan Roy for useful discussions.


\appendix
\def\thesection{\Alph{section}}

\section{Eigenmatrices of association schemes} \label{appendix}
We list the eigenmatrices of the association schemes in Theorem~\ref{thm:as4skew} and in Section~\ref{sec:q}. 
\begin{itemize}
    \item 
Theorem~\ref{thm:as4skew}(1):
\begin{align*}
P&=\left(
\begin{array}{cccc}
 1 & 222 & 333 & 333 \\
 1 & 5 & -3 & -3 \\
 1 & -32 & \frac{1}{2} \left(31+i \sqrt{4699}\right) & \frac{1}{2} \left(31-i \sqrt{4699}\right) \\
 1 & -32 & \frac{1}{2} \left(31-i \sqrt{4699}\right) & \frac{1}{2} \left(31+i \sqrt{4699}\right) \\
\end{array}
\right),\\
Q&=\left(
\begin{array}{cccc}
 1 & 762 & 63 & 63 \\
 1 & \frac{635}{37} & -\frac{336}{37} & -\frac{336}{37} \\
 1 & -\frac{254}{37} & \frac{7}{74} \left(31-i \sqrt{4699}\right) & \frac{7}{74} \left(i \sqrt{4699}+31\right) \\
 1 & -\frac{254}{37} & \frac{7}{74} \left(i \sqrt{4699}+31\right) & \frac{7}{74} \left(31-i \sqrt{4699}\right) \\
\end{array}
\right).
\end{align*}

\item Theorem~\ref{thm:as4skew}(2):
\begin{align*}
P&=\left(
\begin{array}{cccc}
 1 & 590 & 177 & 177 \\
 1 & 5 & -3 & -3 \\
 1 & -40 & \frac{3}{2} \left(3 i \sqrt{35}+13\right) & \frac{3}{2} \left(13-3 i \sqrt{35}\right) \\
 1 & -40 & \frac{3}{2} \left(13-3 i \sqrt{35}\right) & \frac{3}{2} \left(3 i \sqrt{35}+13\right) \\
\end{array}
\right),\\
Q&=\left(
\begin{array}{cccc}
 1 & 826 & 59 & 59 \\
 1 & 7 & -4 & -4 \\
 1 & -14 & \frac{1}{2} \left(13-3 i \sqrt{35}\right) & \frac{1}{2} \left(3 i \sqrt{35}+13\right) \\
 1 & -14 & \frac{1}{2} \left(3 i \sqrt{35}+13\right) & \frac{1}{2} \left(13-3 i \sqrt{35}\right) \\
\end{array}
\right). 
\end{align*}

\item Section~\ref{sec:q} for $(v,k,\lambda)=(36,15,6)$: 
\begin{align*}
P&=\left(
\begin{array}{cccc}
 1 & 15 (w-1) & 35 & 21 (w-1) \\
 1 & 3 (w-1) & -1 & -3 (w-1) \\
 1 & -3 & -1 & 3 \\
 1 & -15 & 35 & -21 \\
\end{array}
\right),Q=\left(
\begin{array}{cccc}
 1 & 35 & 35 (w-1) & w-1 \\
 1 & 7 & -7 & -1 \\
 1 & -1 & 1-w & w-1 \\
 1 & -5 & 5 & -1 \\
\end{array}
\right)
\end{align*}

\item Section~\ref{sec:q} for $(v,k,\lambda)=(45,33,24)$: 
\begin{align*}
P&=\left(
\begin{array}{cccc}
 1 & 33 (w-1) & 44 & 12 (w-1) \\
 1 & 3 (w-1) & -1 & -3 (w-1) \\
 1 & -3 & -1 & 3 \\
 1 & -33 & 44 & -12 \\
\end{array}
\right),Q=\left(
\begin{array}{cccc}
 1 & 44 & 44 (w-1) & w-1 \\
 1 & 4 & -4 & -1 \\
 1 & -1 & 1-w & w-1 \\
 1 & -11 & 11 & -1 \\
\end{array}
\right)
\end{align*}

\item Section~\ref{sec:q} for $(v,k,\lambda)=(64,36,20)$: 
\begin{align*}
P&=\left(
\begin{array}{cccc}
 1 & 36 (w-1) & 63 & 28 (w-1) \\
 1 & 4 (w-1) & -1 & -4 (w-1) \\
 1 & -4 & -1 & 4 \\
 1 & -36 & 63 & -28 \\
\end{array}
\right),Q=\left(
\begin{array}{cccc}
 1 & 63 & 63 (w-1) & w-1 \\
 1 & 7 & -7 & -1 \\
 1 & -1 & 1-w & w-1 \\
 1 & -9 & 9 & -1 \\
\end{array}
\right)
\end{align*}

\item Section~\ref{sec:q} for $(v,k,\lambda)=(96, 76, 60)$: 
\begin{align*}
P&=\left(
\begin{array}{cccc}
 1 & 76 (w-1) & 95 & 20 (w-1) \\
 1 & 4 (w-1) & -1 & -4 (w-1) \\
 1 & -4 & -1 & 4 \\
 1 & -76 & 95 & -20 \\
\end{array}
\right),Q=\left(
\begin{array}{cccc}
 1 & 95 & 95 (w-1) & w-1 \\
 1 & 5 & -5 & -1 \\
 1 & -1 & 1-w & w-1 \\
 1 & -19 & 19 & -1 \\
\end{array}
\right)
\end{align*}

\item Section~\ref{sec:q} for $(v,k,\lambda)=(100, 45, 20)$: 
\begin{align*}
P&=\left(
\begin{array}{cccc}
 1 & 45 (w-1) & 99 & 55 (w-1) \\
 1 & 5 (w-1) & -1 & -5 (w-1) \\
 1 & -5 & -1 & 5 \\
 1 & -45 & 99 & -55 \\
\end{array}
\right),\left(
\begin{array}{cccc}
 1 & 99 & 99 (w-1) & w-1 \\
 1 & 11 & -11 & -1 \\
 1 & -1 & 1-w & w-1 \\
 1 & -9 & 9 & -1 \\
\end{array}
\right)
\end{align*}

\item Section~\ref{sec:q} for $(v,k,\lambda)=(144, 78, 42)$: 
\begin{align*}
P&=\left(
\begin{array}{cccc}
 1 & 78 (w-1) & 143 & 66 (w-1) \\
 1 & 6 (w-1) & -1 & -6 (w-1) \\
 1 & -6 & -1 & 6 \\
 1 & -78 & 143 & -66 \\
\end{array}
\right),\left(
\begin{array}{cccc}
 1 & 143 & 143 (w-1) & w-1 \\
 1 & 11 & -11 & -1 \\
 1 & -1 & 1-w & w-1 \\
 1 & -13 & 13 & -1 \\
\end{array}
\right)
\end{align*}

\item Section~\ref{sec:q} for $(v,k,\lambda)=(153, 57, 21)$: 
\begin{align*}
P&=\left(
\begin{array}{cccc}
 1 & 57 (w-1) & 152 & 96 (w-1) \\
 1 & 6 (w-1) & -1 & -6 (w-1) \\
 1 & -6 & -1 & 6 \\
 1 & -57 & 152 & -96 \\
\end{array}
\right),\left(
\begin{array}{cccc}
 1 & 152 & 152 (w-1) & w-1 \\
 1 & 16 & -16 & -1 \\
 1 & -1 & 1-w & w-1 \\
 1 & -\frac{19}{2} & \frac{19}{2} & -1 \\
\end{array}
\right)
\end{align*}
\end{itemize}

\end{document}